\documentclass[10pt,a4paper]{amsart}
\usepackage[utf8]{inputenc}
\usepackage{amsmath}
\usepackage{amsfonts}
\usepackage{amssymb}
\usepackage{graphicx}
\usepackage[left=2cm,right=2cm,top=2cm,bottom=2cm]{geometry}
\newtheorem{Theorem}{Theorem}[section]

 \newtheorem{Cor}{Corollary}
 
 \newtheorem{Proposition}[Theorem]{Proposition}
 \newtheorem{ex}[Theorem]{Example}
 \newtheorem{Remark}[Theorem]{Remark}
 \numberwithin{equation}{subsection}
 
\begin{document}
\title[DYNAMICS OF THE INDUCED SHIFT MAP]{DYNAMICS OF THE INDUCED SHIFT MAP}
\author{Puneet Sharma and Anima Nagar}
\email{puneet.iitd@yahoo.com, anima@maths.iitd.ac.in}%
\footnote{This work is a part of a Ph.D. thesis sucessfully defended in 2011.}
\thanks{The first author thanks CSIR  for financial support.}%

\begin{abstract}

In this article, we compare the dynamics of the shift map and its
induced counterpart on the hyperspace of the shift space. We show
that many of the properties of induced shift map can be easily
demonstrated by appropriate sequences of symbols. We compare the
dynamics of the shift system $(\Omega, \sigma)$ with its induced
counterpart $(\mathcal{K}(\Omega),\overline{\sigma})$, where
$\mathcal{K}(\Omega)$ is the hyperspace of all nonempty compact
subsets of $\Omega$. Recently, such comparisons have been studied a
lot for general spaces. We continue the same study in case of shift
spaces, and bring out the significance of such a study in terms of
sequences.

We compare the mixing properties,  denseness of periodic points,
various forms of sensitivity and expansivity of the shift map and
its induced counterpart. In particular, we show their equivalence in
case of the full shift.  We also look into the special case of
subshifts of finite type, and in particular prove that the
properties of weakly mixing, mixing and sensitivity are equivalent
in both systems. And in the case of any general subshift, we show
that the concept of cofinite sensitivity is equivalent in both
systems and for transitive subshifts, cofinite sensitivity is
equivalent to syndectic sensitivity. In the process we prove that
all sturmian subshifts are cofinitely sensitive.
\end{abstract}
\maketitle

\section{INTRODUCTION}

Many times  dynamical systems are studied by discretizing both time
and state space. The basic idea involves in taking a partition of
the state space into finite number of regions, each of which can be
labelled with some symbol. Time is then discretized by taking
iterates of all points in the space. Each itinerary in the state
space then corresponds to an infinite sequence of symbols, where the
symbols are the labels  of the region in the partition given by the
trajectory of the point. This `Symbolic Dynamics' though gives an
approximation of the actual orbits, but is very useful in capturing
the essence of any dynamics.

Symbolic systems are important classes of dynamical systems and have
great applicability to topological dynamics and ergodic theory.
Their equivalence with many topological dynamical systems and simple
computational structure makes them  an important class of dynamical
systems. They have also been used to approximate various natural
processes and predict their long term behavior. Further, it has been
seen that most of the dynamical systems, observed in Nature, are
collective(set valued) dynamics of many units of individual systems.
In particular, the asymptotic behavior of the iterates of any
non-empty subset of the space becomes an important study. Hence,
there is a strong need to develop a relation between the dynamics on
the base space and the hyperspace(space of subsets). Such a study
can help in understanding the combined dynamics on systems, which on
an individual basis may not be that interesting. This has lead to
the study of `set-valued dynamics'. Roman Flores \cite{ro}, Banks
\cite{ba}, Liao, et al \cite{th}, Sharma and Nagar \cite{sn1,sn2}
have given a comparison of individual dynamics and set-valued
dynamics. On the other hand, it has been observed that dynamical
systems can be better studied via symbolic dynamics \cite{ki, ml,
py}. Also in \cite{fu}, it is shown that any dynamical system can be
realized as a subshift of some shift space. 

Some recent studies of dynamical systems, in branches of
engineering and physical sciences, have revealed that the
underlying dynamics is set valued or collective, instead of the
normal individual kind which is usually studied. Some recent
studies in Population Dynamics, consider population as local
subpopulation in discrete habitat patches, with independent
dynamics. This initiates the study of metapopulation dynamics (
see \cite{rr}). In Chemical Physics, the individual dynamics of
the electron and the nuclei are combined to stimulate the dynamics
of large molecular systems containing thousands of atoms ( see
\cite{sg}). In Atmospheric Sciences, the perturbation of waves is
studied as a combined effect of the near-surface,
intermediate-level and tropopause-level perturbations upon flow
development ( see \cite{dd}). In Mechanical Engineering recently,
lane keeping controllers have been specifically designed, so that
they can be coupled with steering force feedback for better
maintenance of lane position in absence of driver steering
commands. Artificial damping is further injected to make the
combined effect of the system stable, ensuring risk free and safe
driving ( see \cite{srcg}).

With these varieties of dynamics observed, there arises the need of
a topological treatment of such collective dynamics. Also, the
evolution of trajectories of a chaotic dynamical system is
equivalent to symbolic dynamics in an appropriate symbol system.
Hence, there is a strong need to develop a relation between the
dynamics on the shift system and its induced counterpart on its
hyperspace.

In this article, we study the relations between the dynamical
behavior of the shift $(\Omega, \sigma)$ and its induced counterpart
$(\mathcal{K}(\Omega), \overline{\sigma})$. We consequently show
that many of the chaotic properties of $(\mathcal{K}(\Omega),
\overline{\sigma})$ can be easily exhibited by the sequences in
$\Omega$.


\section{THEORETICAL PRELIMINARIES}

We now introduce some basics from dynamical systems, hyperspace
topologies and symbolic dynamics.

\subsection{Dynamical Systems}

Let $(X, \tau)$ (resp. $(X,d)$) be a topological (resp. metric)
space and let $f : X \rightarrow X$ be a continuous function. The
pair $(X,f)$ is referred as a dynamical system. We state some
dynamical properties here, though we refer to
\cite{ak,bc,bs,de,subru} for more details.

A point $x \in X$ is called \textit{periodic} if $ f^n (x)=x$ for
some positive integer $n$, where $f^n = f\circ f \circ f\circ
\ldots \circ f$ ($n$ times). The least such $n$ is called the
\textit{period} of the point $x$.  If there exists a $\delta > 0$
  such that for every $x \in X$ and for each $\epsilon > 0$ there
exists $y \in X$ and a positive integer $n$ such that $d(x,y) <
\epsilon$ and $ d(f^n(x), f^n(y)) > \delta$, then $f$ is said to
be \textit{sensitive} ($\delta$-sensitive). The constant $\delta$
is called the \textit{sensitivity constant} for $f$.  $f$ is said
to be \textit{cofinitely sensitive}, if there exists $\delta
>0$ such that the set of instances
$N_f (U, \delta) = \{n \in \mathbb{N} :$ there exist   $ y, z \in
U$ with $d(f^n(y), f^n(z)) > \delta \}$ is cofinite.  $f$ is
called \textit{syndetically sensitive} if there exists $\delta >
0$ with the property that for every $\epsilon$-neighborhood $U$ of
$x$,  $N_f (U, \delta)$ is syndetic. In general,

\centerline{$\textit{cofinitely sensitive} \Rightarrow
\textit{syndetically sensitive} \Rightarrow  \textit{sensitive}$}

  $f$ is \textit{Li-Yorke sensitive} if
there exists $\delta >0$ such that for each $x \in X$ and for each
$\epsilon >0$, there exists $y \in X$ with $d(x,y) < \epsilon$
such that $\liminf \limits_{n \rightarrow \infty} d(f^n(x),
f^n(y)) =0$ but $\limsup \limits_{n \rightarrow \infty} d(f^n(x),
f^n(y)) > \delta$. A very strong form of sensitivity is
expansivity. $f$ is called \textit{expansive} ($\delta$-expansive)
if for any pair of distinct elements $x,y \in X$, there exists $k
\in \mathbb{N}$ such that $d(f^k(x), f^k(y)) > \delta$.

$f$ is called \textit{transitive} if for any pair of non-empty
open sets $U,V$ in $X$, there exist a positive integer $n$ such
that $f^n(U) \bigcap V \neq \phi$, and is called \textit{totally
transitive} if $f^n$ is transitive for each $n \in \mathbb{N}$.
$f$ is called \textit{weakly mixing} if $f \times f$ is
transitive.   $f$ is called \textit{mixing} or
\textit{topologically mixing} if for each pair of non-empty open
sets $U,V$ in $X$, there exists a positive integer $k$ such that
$f^n(U) \bigcap V \neq \phi$ for all $n \geq k$.  $f$ is called
\textit{locally eventually onto (leo)} if for each non-empty open
set $U$, there exists a positive integer $k \in \mathbb{N}$ such
that $f^k(U) = X$. Among the above topological properties, the
following relation holds,

\centerline{$\textit{leo} \Rightarrow \textit{topological mixing}
\Rightarrow \textit{weakly mixing} \Rightarrow \textit{totally
transitivity} \Rightarrow \textit{transitivity}$}

\subsection{Hyperspace Topologies}

For a Hausdorff space $(X, \tau)$, a hyperspace $(\mathcal{K}(X),
\Delta)$ comprises of all nonempty compact subsets of $X$ endowed
with the topology $\Delta$, where the topology $\Delta$ is defined
using the topology $\tau$ of $X$.  The topology $\Delta$,  that we
consider here will be either the Vietoris topology or the Hausdorff
Metric topology (when $X$ is a metric space). We briefly describe
these topologies.

Define, $ \langle U_1, U_2, \ldots, U_k \rangle$ = $ \{E \in
\mathcal{K}(X) :E \subseteq \bigcup \limits_{i=1}^k U_{i}$ and E
$\bigcap U_{i} \neq \phi \textrm{ } \forall i \}$. The topology
generated by the collection of all such sets, where $k$ varies over
all possible natural numbers and $U_i$ varies over all possible open
subsets of $X$, is known as the \textit{Vietoris topology}.

For a metric space $(X,d)$ and for any two non-empty compact
subsets $A_1, A_2$ of $X$, define, $ d_H (A_1, A_2) = \inf \{
\epsilon >0 : A \subseteq S_{\epsilon} (B) \text{ and } B
\subseteq S_{\epsilon} (A) \} $ where $S_{\epsilon}(A) = \bigcup
\limits_{x \in A} S(x, \epsilon)$ and $S(x, \epsilon) = \{ y \in X
: d(x,y)< \epsilon \}$. $d_H$  is a metric on $\mathcal{K}(X)$ and
is known as the \textit{Hausdorff metric}, which generates  the
\textit{Hausdorff metric topology} on $\mathcal{K}(X)$.

It is known that $\mathcal{K}(X)$ is compact if and only if $X$ is
compact and in this case, the Hausdorff metric topology is
equivalent to the Vietoris topology. Also, it is known that the
collection of finite sets is dense in $\mathcal{K}(X)$. We can talk
of these topologies for any subspace $\Psi$ of $\mathcal{K}(X)$. See
\cite{be,mi} for details.

\subsection{Symbolic Dynamics}

We study the sequence space generated by a symbol set
$\mathcal{A}$, where $|\mathcal{A}|$ may be finite or infinite. In
general, we study the sequence space $\mathcal{A}^{\mathbb{N}}$ or
$\mathcal{A}^{\mathbb{Z}}$.

Let $\mathcal{A}$ be a discrete alphabet set ($|\mathcal{A}|$ may
be finite or infinite).  Let $\Sigma_{\mathcal{A}} =
\mathcal{A}^{\mathbb{N}}$ be the space of all infinite sequences
over $\mathcal{A}$. Define $D_1 : \Sigma_{\mathcal{A}} \times
\Sigma_{\mathcal{A}} \rightarrow \mathbb{R}^{+}$  as,

\centerline{$D_1 (\bar{x},\bar{y}) = \sum \limits_{i= 0}^{\infty}
\frac{\delta(x_i, y_i)}{2^{i}}$}

where, $\bar{x} = (x_i), \bar{y} = (y_i)$ and $\delta$ is the
discrete metric.

Then, $D_1$ defines a metric which generates the product topology
on $\mathcal{A}^{\mathbb{N}}$. Similarly, $D_1 (\bar{x},\bar{y}) =
\sum \limits_{i= - \infty}^{\infty} \frac{\delta(x_i,
y_i)}{2^{|i|}}$ defines a metric which generates the product
topology on $\mathcal{A}^{\mathbb{Z}}$, the space of all bifinite
sequences over $\mathcal{A}$.

It can be seen that the set $[i_0 i_1 \ldots i_k] = \{ (x_n) :
x_r=i_r, 0 \leq r \leq k \}$ is a clopen set in
$\Sigma_{\mathcal{A}}$, and is referred to as a \textit{cylinder
set}. Any open set, in $\mathcal{A}^{\mathbb{N}}$, is a countable
union of such sets. Consequently, the cylinder sets form a basis
for the product topology on $\Sigma_{\mathcal{A}}$. The
\textit{shift} (left shift) operator, defined as
 $\sigma(x_0 x_1 \ldots) = x_1 x_2 x_3 \ldots $
 is known to be continuous when the space is
equipped with the metric $D_1$. We refer the  system
$(\mathcal{A}^{\mathbb{N}}, \sigma)$ as the \textit{full shift
(shift)} space.

When $\mathcal{A}$ is a finite set, $\mathcal{A}^{\mathbb{N}}$
(resp. $\mathcal{A}^{\mathbb{Z}}$) is a compact metrizable space.
Let $\Sigma \subseteq \Sigma_{\mathcal{A}}$ be a closed
$\sigma$-invariant subset of $\Sigma_{\mathcal{A}}$.  If there
exists a finite collection of words (finite strings) that are
forbidden in any sequence in $\Sigma$, then the subsystem
$(\Sigma, \sigma)$ is called a \textit{subshift of finite type}.
Every subshift of finite type can be represented by a $\{0,1\}$
square matrix. In such a case, the matrix $M$ is called the
\textit{transition matrix} for the space $\Sigma_{M}$.

It may be noted that the subsystems ``subshifts of finite type" can
be considered also when the symbol set $\mathcal{A}$ is infinite.
 We will, however, not consider such cases.

Let $M$ be a transition matrix. $M$ is said to be
\textit{irreducible} if for every pair of indices $i$ and $j$
there is an $l >0$ with $(M^l)_{ij} >0$. Fix an index $i$ and let
$p(i) = \gcd \{ l : (M^l)_{ii}>0 \}$. This is called the
\textit{period} of the index $i$. When $M$ is irreducible, period
of every index is same and is called the \textit{period of $M$}.
If the matrix has period one, it is said to be \textit{aperiodic}
(see \cite{ki, ml, py}). Also, $\Sigma_M$ is transitive if and
only if $M$ is irreducible. $M$ is irreducible and aperiodic if
and only if there exists $r \in \mathbb{N}$ such that for all $k
\geq r$, $M^k$ is strictly positive. $\Sigma_M$ is topological
mixing if and only if $M$ is irreducible and aperiodic. Also, from
\cite{subru}  the following are equivalent.

1. $\Sigma_M$ is totally transitive

2. $\Sigma_M$ is weakly mixing.

3. $\Sigma_M$ is topological mixing.

We now discuss the case when the alphabets in $\mathcal{A}$ are
elements of some metric space, i.e. when $(\mathcal{A},d)$ is a
general metric space. Then, $\mathcal{A}^{\mathbb{N}}$ equipped
with the metric

\centerline{$D_2 (\bar{x},\bar{y}) = \sum \limits_{i= 0}^{\infty}
\frac{1}{2^{i}} \frac{d(x_i, y_i)}{1 + d(x_i, y_i)}$}

generates the product topology on $\mathcal{A}^{\mathbb{N}}$.

For any set of symbols $\mathcal{A}$,  let $\mathcal{A}^\mathbb{N}$
be endowed with metric $D$ (which is $D_1$ or $D_2$, depending on
the space $\mathcal{A}$). In all such cases,
$(\mathcal{A}^{\mathbb{N}}, \sigma)$ is a dynamical system.

From \cite{fu} we see that if $(X,f)$ be a compact dynamical system,
and  $\Sigma = \{ (x, f(x),\\ f^2(x), \ldots, f^n(x), \ldots) : x
\in X \}$, then $\Sigma$ is a shift invariant subset of
$X^{\mathbb{N}}$ and thus, $(\Sigma, \sigma)$ is a subsystem of the
full shift $X^{\mathbb{N}}$. Also, if we define, $ \phi : X
\rightarrow \Sigma $

 \centerline{$ \phi(x)
= (x, f(x), f^2(x), \ldots, f^n(x), \ldots)$}

Then, $\phi$ is one-one, onto and continuous function satisfying
the relation $ \phi \circ f = \sigma \circ \phi$. Thus, the system
$(X,f)$ is conjugate to the system $(\Sigma, \sigma)$. This leads
to the observation that for any dynamical system $(X,f)$, there
exists $\Sigma \subseteq X^{\mathbb{N}}$ such that the system
$(X,f)$ is conjugate to the system $(\Sigma, \sigma)$.

Also, \cite{subru} proves that a point in  $ \Sigma$ is a point of
sensitivity for the system $(\Sigma, \sigma)$ if and only if it is
not isolated. Though, in \cite{subru}, this
 has been proved for the case of a discrete alphabet
set, it can be easily established for the case when the alphabet
set is any general metric space.

Hence, to study any dynamical system, it is sufficient to study a
subsystem of an appropriate symbolic  system. Henceforth, we shall
constrain ourselves to the subsystems $(\Omega, \sigma)$ of the
symbolic space $(\mathcal{A}^{\mathbb{N}}, \sigma)$, where the
symbol set $\mathcal{A}$ comprises of the points in the metric
space $(\mathcal{A},d)$, where $d$ is the discrete metric in case
$\mathcal{A}$ is a discrete alphabet set.

\section{MAIN RESULTS}

\subsection{ We first consider the case when $(\Omega, \sigma)$ is the full
shift, i.e. $\Omega = \mathcal{A}^{\mathbb{N}}$.}

It has been shown that if  $(X,f)$ has dense set of periodic points,
$(\mathcal{K}(X), \overline{f})$ also has the same \cite{ba,sn1}.
We, prove the same result in terms of sequences.

\begin{Proposition}
$(\mathcal{K}(\Omega), \overline{\sigma})$ has dense set of periodic
points.
\end{Proposition}
\begin{proof}
Let $\mathcal{U} \subset \mathcal{K}(\Omega)$ be any open set. As
the set of finite  sequences in $\Omega$ is dense in
$\mathcal{K}(\Omega)$, let $\{\bar{x}^1, \bar{x}^2, \ldots,
\bar{x}^s \} \in \mathcal{U}$, where $\bar{x}^j = (x^j_n)$. Then,
there exists $r \in \mathbb{N}$ such that $S(\{\bar{x}^1, \bar{x}^2,
\ldots, \bar{x}^s \}, \frac{1}{2^r}) \subset \mathcal{U}$.

Let $\bar{y}^j = x^j_0x^j_1 \ldots x^j_{r}x^j_0x^j_1 \ldots
x^j_{r}x^j_0x^j_1 \ldots x^j_{r} \ldots \ $ for $1 \leq j \leq s$.

Each $\bar{y}^j$ is periodic under $\sigma$ with period $r+1$ and
$D(\bar{x}^j, \bar{y}^j) < \frac{1}{2^r}$.

Thus, $\{\bar{y}^1, \bar{y}^2, \ldots, \bar{y}^s \}$ is a periodic
point in $\mathcal{U}$ with period $r+1$.
\end{proof}

The system $(\Omega, \sigma)$ is  locally eventually onto for any
alphabet set $\mathcal{A}$. And we also have

\begin{Proposition}
$(\mathcal{K}(\Omega), \overline{\sigma})$ is locally eventually
onto.
\end{Proposition}
\begin{proof}
Let $\mathcal{U}$ be a non empty open set in $\mathcal{K}(\Omega)$.
As set of finite sequences is dense in $\mathcal{K}(\Omega)$, let
$\{\bar{x}^1, \bar{x}^2, \ldots, \bar{x}^s \} \in \mathcal{U}$ where
$\bar{x}^i = (x^i_j)_{j \in \mathbb{N}}$ . Thus, there exists $r \in
\mathbb{N}$ such that $S(\{\bar{x}^1, \bar{x}^2, \ldots, \bar{x}^s
\}, \frac{1}{2^r}) \subset \mathcal{U}$.

We shall show that $\overline{\sigma}^r(S(\{\bar{x}^1, \bar{x}^2,
\ldots, \bar{x}^s \}, \frac{1}{2^r})) = \mathcal{K}(\Omega)$. Let
$K \in \mathcal{K}(\Omega)$. Let $A_i = \{ x^i_0 x^i_1 x^i_2
\ldots x^i_r \bar{a} : \bar{a} \in K \}$. Then, $A_i$ is a compact
subset of $\Omega$. Let $M = \bigcup \limits_{i=1}^s A_i$. Then,
$M \in S(\{\bar{x}^1, \bar{x}^2, \ldots, \bar{x}^s \},
\frac{1}{2^r})$. Also, $\overline{\sigma}^{r+1}(M) = K$. As $K \in
\mathcal{K}(\Omega)$ was arbitrary,
$\overline{\sigma}^{r+1}(S(\{\bar{x}^1, \bar{x}^2, \ldots,
\bar{x}^s \}, \frac{1}{2^r})) = \mathcal{K}(\Omega)$. Hence the
result.
\end{proof}

This can also be seen in a general case, but our proof is
specialized for sequences. The case of sensitivity is not very
simple in general(see \cite{sn2}). But, we observe that $(\Omega,
\sigma)$ is sensitive for any discrete alphabet set $\mathcal{A}$,
with sensitivity constant $\frac{1}{2}$. Similarly,

\begin{Proposition}
$(\mathcal{K}(\Omega), \overline{\sigma})$ is sensitive with
sensitivity constant $\frac{1}{2}$.
\end{Proposition}
\begin{proof}
To establish sensitivity on the induced system, it is sufficient
to prove sensitivity on the collection of all finite subsets,
since finite sets are dense in $\mathcal{K}(\Omega)$.

Let $A = \{\bar{x}^1, \bar{x}^2, \ldots, \bar{x}^s \}$ be a finite
subset where $\bar{x}^i = x^i_0 x^i_1 x^i_2 \ldots$. Let $\epsilon
>0$ and let $S_{\epsilon}(A)$ be an $\epsilon$-neighborhood of $A$.
Let $n \in \mathbb{N}$ such that $\frac{1}{2^n} < \epsilon$. Pick $x
\in \mathcal{A}$ such that $x \neq x^1_{n+1}$.

Let $B = \{\bar{y}^1, \bar{y}^2, \ldots, \bar{y}^s\}$ where
$\bar{y}^i = x^i_0 x^i_1 x^i_2 \ldots x^i_n x x x \ldots$.

Then, $B \in S_{\epsilon}(A)$ and $\overline{\sigma}^{n+1}(B) =
\{x x x \ldots\}$.

\centerline{Consequently, $d_H(\overline{\sigma}^{n+1}(A),
\overline{\sigma}^{n+1}(B))
> \frac{1}{2}$.}
\end{proof}

For sensitivity, we use the discrete metric on $\mathcal{A}$ to
establish the sensitivity on the hyperspace. However, if
$\mathcal{A}$ is equipped with a metric $d$ and $diam(\mathcal{A})=
r$, a similar proof establishes the sensitivity for the induced map
on the hyperspace, with sensitivity constant $\frac{r}{2(r+1)}$.

\begin{Proposition}
For any alphabet set $\mathcal{A}$, containing at least two
elements,

1. $(\Omega, \sigma)$ is Li-Yorke sensitive.

2. $(\mathcal{K}(\Omega), \overline{\sigma})$ is Li-Yorke
sensitive.
\end{Proposition}
\begin{proof}

 1. We note that when $\mathcal{A}$ is a discrete alphabet set,
$(\Omega, \sigma)$ is Li-Yorke sensitive. For the sake of completion
we include the proof of the case when $(\mathcal{A}, d)$ is a metric
space.

Let $ \bar{x} =(x_0 x_1 \ldots ) \in \mathcal{A}^{\mathbb{N}}$ and
let $S_{\epsilon}(\bar{x})$ be an $\epsilon$-neighborhood of
$\bar{x}$. Let $n_0 \in \mathbb{N}$ such that $\frac{1}{2^{n_0}} <
\epsilon$. For each $n \in \mathbb{N}$, choose $x^*_{2^n} \in X$
such that $ \frac{d(x^*_{2^n}, x_{2^n})}{ 1 + d(x^*_{2^n},
x_{2^n})}
> \frac{r}{2(1+r)}$. where  $diam(\mathcal{A})
>r, \ r \in \mathbb{R}$. Replace $2^n$-th entries( i.e. $x_{2^n}$)
in the sequence $\bar{x}$ for $n \geq n_0$ by $x^*_{2^n}$(keeping
all others same) to obtain new sequence $\bar{y}$. Consequently, $
\bar{y} \in S_{\epsilon}(\bar{x})$. Further,
$D(\sigma^{2^k+1}(\bar{x}),\sigma^{2^k+1}(\bar{y})) <
\frac{1}{2^k}$ and $D(\sigma^{2^k}(\bar{x}),\sigma^{2^k}(\bar{y}))
> \frac{r}{2(1+r)}$ for infinitely many $k$. Thus, the space
$(\mathcal{A}^{\mathbb{N}}, \sigma)$ is Li-Yorke sensitive.

2. We now prove that the system
$(\mathcal{K}(\mathcal{A}^{\mathbb{N}}), \overline{\sigma})$ is
Li-Yorke sensitive. It suffices to show the same in case of the
alphabet set being equipped with a metric $d$, and
$diam(\mathcal{A}) > r, \ r \in \mathbb{R}$.

 We now prove that $(\mathcal{K}(\Omega), \overline{\sigma})$ is
Li-Yorke sensitive. Let $K_1 \in \mathcal{K}(\Omega)$ and let
$S_{\frac{1}{m}}(K_1)$ ($m > 1$) be a neighborhood of $K_1$ in the
hyperspace. Since $K_1$ is totally bounded and periodic sequences
are dense in $\Omega$, let $W_m = \{w_{1,m}, w_{2,m}, \ldots,
w_{t_m,m}\}$ be the finite set of words such that
$\frac{1}{m}$-ball around the sequences $\{ www \ldots : w \in
W_m\}$ covers $K_1$. Let each $w_{i,m}$ be of length $L^i_m$ and
let $L_m$ be the least common multiple of all such lengths.

As $\sigma^{L_m+1}(K_1)$ is compact, we construct the set $W_{m+1}$
 of finitely many words such that $\frac{1}{m+2}$-ball
around the sequences $\{ www \ldots : w \in W_{m+1}\}$ covers the
set $\sigma^{L_m+1}(K_1)$. Let $L_{m+1}$ be the least common
multiple of lengths of all words in $W_{m+1}$.

Inductively, define the set $W_{m+i}$ as the set of finitely many
words such that $\frac{1}{m+2^i}$-ball around the sequences $\{ www
\ldots : w \in W_{m+i}\}$ covers the set $\sigma^{L_m + L_{m+1}+
\ldots+ L_{m+i-1}+i}(K_1)$.

Let $\bar{y}=(y_0 y_1 \ldots y_n \ldots) \in K_1$. For each $i$,
choose $x_i \in \mathcal{A}$ such that $d(x_i, y_{L_m + L_{m+1}+
\ldots + L_{m+i}+i+1}) \geq \frac{r}{2}$.

Construct the set $K_2 = \{w_0 w_0 \ldots w_0 x_0 w_1 w_1 \ldots
w_1 x_1 \ldots : w_i \in W_{m+i} \}$, where the repetitions of
$w_i$ are made till length $L_{m+i}$ is achieved and  $w_i$ varies
over all possible entries in $W_{m+i}$.

We see that $K_2 \in \mathcal{K}(\Omega)$.

Let $\bar{z}$ be any limit point of ${K_2}$. Then, there exists a
sequence $(\bar{z}_n) \in K_2$ such that $\bar{z}_n \rightarrow
\bar{z}$. But, in any sequence $\bar{z}_n$, the first $L_m$ entries
 are repetitions of words from $W_m$, which has finite number of
choices. Hence, there exists a subsequence $(\bar{z}_{n_k})$ and a
word $w_0* \in W_m$ such that first $L_m$ entries are repetitions
of $w_0*$ for each member of the subsequence $(\bar{z}_{n_k})$.

Also, $\bar{z}_{n_k} \rightarrow \bar{z}$. And so, in $\bar{z}$,
the first $L_m$ entries   are repetitions of $w_0*$ and the ${L_m
+ 1}^{th}$ entry will be $x_0$. Again, the next $L_{m+1}$ entries
in each $\bar{z}_{n_k}$ are repetition of words from $W_{m+1}$,
which have finite number of choices. Hence, there exists a
subsequence $\bar{z}_{n_{k_r}}$ in which the next $L_{m+1}$
entries are repetitions of a word $w_1* \in W_{m+1}$.
Consequently, the next $L_{m+1}$ entries in $\bar{z}$ are
repetitions of the word $w_1*$. Inductively, we see that $\bar{z}$
is solely of the form of sequences in $K_2$, implying that
$\bar{z} \in K_2$ i.e. $K_2$ is closed.

Then, $K_2 \in S_{\frac{1}{m}}(K_1)$ with $\limsup \limits_{n
\rightarrow \infty} D_H(\overline{\sigma}^n(K_1),
\overline{\sigma}^n(K_2)) \geq \frac{r}{2(1+r)}$ and $\liminf
\limits_{n \rightarrow \infty} D_H(\overline{\sigma}^n(K_1),
\overline{\sigma}^n(K_2)) =0$. Thus, $(\mathcal{K}(\Omega),
\overline{\sigma})$ is Li-Yorke sensitive. \end{proof}

\subsection{ We now consider the case when $(\Omega, \sigma)$ is a subshift of
finite type described via the transition matrix $M$,}
\textbf{where,}

\vskip 0.5cm

\centerline{ $M = \left(%
\begin{array}{ccccc}
  m_{11} & m_{12} & . & . & m_{1n} \\
  m_{21} & m_{22} & . & . & m_{2n} \\
  . & . & . & . & . \\
  . & . & . & . & . \\
  m_{n1} & m_{n2} & . & . & m_{nn} \\
\end{array}%
\right)$}

\textbf{with each $m_{ij} = 0$ or $1$.}

Let $S_2 = \{ (i_1,i_2) : 1\leq i_1,i_2 \leq n \}$. Let $M^{*2}$ be
a matrix indexed by entries of $S_2$ defined as, $M^{*2}_{(i_1,
i_2), (j_1, j_2)} = m_{i_1 j_1} m_{i_2 j_2}$

Then $M^{*2}$ is a square matrix of order $n^2$, given as,

\centerline{{$M^{*2} = \left(%
\begin{array}{cccccccc}
  m_{11} m_{11} & . & . & m_{1n} m_{1n}| & . & . & . & m_{1n} m_{1(n-1)} \\
  m_{21} m_{21} & . & . & m_{2n} m_{2n}| & . & . & . & . \\
  . & . & . & . & . & . & . & . \\
    \underline{m_{n1} m_{n1}} & \underline{.} & \underline{.} & \underline{m_{nn} m_{nn}}| & . & . & . & . \\
    . & . & . & . & . & . & . & . \\
  . & . & . & . & . & . & . & . \\
  . & . & . & . & . & . & . & . \\
  m_{n1} m_{(n-1)1} & . & . & . & . & . & . & m_{nn} m_{(n-1)(n-1)} \\
\end{array}%
\right)$}}

When $\Omega$ comprises of $n$ elements, the transition matrix $M$
provides the dynamical behavior of the system, giving the details of
how  two given regions, labelled by  distinct alphabets, of the
system interact. In the matrix generated above, $M^{*2}_{(i_1, i_2),
(j_1, j_2)}$ gives the details of simultaneous interaction of the
$i_1-th, j_1-th$ and $i_2-th, j_2-th$ regions of the original system
respectively. As the entries vary over all possible combinations,
the matrix determines the dynamics of $(\Omega \times \Omega, \sigma
\times \sigma)$. Thus, the system $(\Omega \times \Omega, \sigma
\times \sigma)$ can be embedded into certain symbolic dynamical
system of $n^2$ symbols. This is similar to the concept of ``higher
block codes" discussed in \cite{ml}.

Similarly, let  $S_k = \{ (i_1,i_2, \ldots i_k) : 1\leq i_1,i_2,
\ldots i_k \leq n \}$. Let $M^{*k}$ be a matrix indexed by entries
of $S_k$ defined as, $M^{*k}_{(i_1, i_2, \ldots i_k), (j_1, j_2,
\ldots j_k)} = m_{i_1 j_1} m_{i_2 j_2} \ldots m_{i_k j_k}$

Then $M^{*k}$ is a square matrix of order $n^k$, given as

\centerline{{$M^{*k} = \left(%
\begin{array}{cccccccc}
  m_{11} m_{11}\ldots m_{11} & . & . & m_{1n} m_{1n} \ldots m_{1n}| & . & . & . & m_{1n} m_{1n} \ldots m_{1n} m_{1(n-1)} \\
  m_{21} m_{21}\ldots m_{21} & . & . & m_{2n} m_{2n}\ldots m_{2n}| & . & . & . & . \\
  . & . & . & . & . & . & . & . \\
    \underline{m_{n1} m_{n1}\ldots m_{n1}} & \underline{.} & \underline{.} & \underline{m_{nn} m_{nn} \ldots m_{nn}}| & . & . & . & . \\
    . & . & . & . & . & . & . & . \\
  . & . & . & . & . & . & . & . \\
  . & . & . & . & . & . & . & . \\
  m_{n1} m_{n1} \ldots m_{n1} m_{(n-1)1} & . & . & . & . & . & . & m_{nn} m_{nn} \ldots m_{nn} m_{(n-1)(n-1)} \\
\end{array}%
\right)$}}

Arguing as before, the matrix $M^{*k}$ determines the simultaneous
interaction of a set of $k$ regions of the original space, with
another set of $k$ regions of the same space. The matrix $M^{*k}$
hence determines the dynamics of $(\Omega \times \Omega \times
\ldots \times \Omega, \sigma \times \sigma \times \ldots \times
\sigma)$, where the cartesian product is taken $k$ number of
times.

It can be observed that the first $n \times n$ block in $M^{*k}$
is just $M$.

The Vietoris topology is generated by sets of the form $<U_1,U_2,
\ldots, U_n>$. Hence, while studying the hyperspace under Vietoris
topology, it is sufficient to study the simultaneous behavior of
finitely many regions of the original space. Consequently, the
behavior of the matrices $M^{*k}$ is sufficient to study the
dynamics of the hyperspace $(\mathcal{K}(\Omega),
\overline{\sigma})$.

For a dynamical system $(X,f)$, $f$ is weakly mixing if and only
if for every $k \geq 2$, $\underbrace{f \times f \times \ldots
\times f}_{k \text{ } times}$ is transitive \cite{th}. Thus,
$\sigma$ is weakly mixing if and only if  $\underbrace{\sigma
\times \sigma \times \ldots \times \sigma}_{ k \text{ } times}$ is
transitive for each $ k \geq 2$, and by the discussion above, if
and only if $M^{*k}$ is irreducible for $ k \geq 2$.

This proves that on a subshift of finite type, described via a
transition matrix $M$, $\sigma$ is weakly mixing if and only if
for every $k \geq 2$, $M^{*k}$ is irreducible. And also
establishes that the following are equivalent :

1.  $(\Omega,\sigma)$ is weakly mixing

2. $(\mathcal{K}(\Omega), \overline{\sigma})$ is weakly mixing

3.  $(\mathcal{K}(\Omega),\overline{\sigma})$ is transitive

Similarly, the following are equivalent :

1.  $(\Omega,\sigma)$ is topologically mixing

2. $(\mathcal{K}(\Omega), \overline{\sigma})$ is topologically
mixing

Also $M$ is irreducible and aperiodic if and only if  $M^{*k}$ is
irreducible and aperiodic for all $k \in \mathbb{N}$, establishing
that $(\Omega,\sigma)$ is topologically mixing if and only if
$(\mathcal{K}(\Omega), \overline{\sigma})$ is topologically
mixing.

Transitivity of the induced system $(\mathcal{K}(\Omega),
\overline{\sigma})$ is equivalent to weak mixing of the original
system $(\Omega,\sigma)$. But, if the system $(\Omega,\sigma)$ is
transitive, the induced system $(\mathcal{K}(\Omega),
\overline{\sigma})$ may fail to be transitive. This can be
illustrated by a simple counterexample :

\begin{ex}
Let $\Sigma_M$ be the subshift of finite type given by the matrix

\centerline{$M = \left(%
\begin{array}{cccc}
  0 & 0 & 1 & 1 \\
  0 & 0 & 1 & 1 \\
  1 & 1 & 0 & 0 \\
  1 & 1 & 0 & 0 \\
\end{array}%
\right)$}

Then, $M$ is irreducible and hence generates a transitive subshift.
However $M^{*k}$ is not irreducible for $k \geq 2$. Hence, the
induced map is not transitive.
\end{ex}

\begin{Remark}
The above results involving transitivity, weakly mixing and mixing
are known in the case of any general dynamical system $(X,f)$. See
\cite{ba,ro,sn1}
\end{Remark}

Whereas for sensitivity we can observe that
\begin{Proposition}
$(\Omega,\sigma)$ is sensitive if and only if
$(\mathcal{K}(\Omega), \overline{\sigma})$ is sensitive.
\end{Proposition}
\begin{proof}
Let $\Omega$ be a subshift of finite type and let  $(\Omega,
\sigma)$ be sensitive. We now show that $(\mathcal{K}(\Omega),
\overline{\sigma})$ is sensitive. In order to prove this, we show
that the induced map, $\overline{\sigma}$, is sensitive on all
finite subsets of $\Omega$, with a uniform sensitivity constant.
Without loss of generality, we  assume that all forbidden words for
$\Omega$ have the same length, say $M$.

Let $A = \{ \bar{x}^1,\bar{x}^2,\ldots,\bar{x}^r\}$ be a finite
subset of $\Omega$ and let $[x^i_0 x^i_1 x^i_2 \ldots x^i_n]$ be a
neighborhood of $\bar{x}^i$, $i=1,2,\ldots,r$. We show that
$\overline{\sigma}$ is sensitive on $A$ with sensitivity constant
$\frac{1}{2^{M+1}}$.

For each $i$, in the sequence $\bar{x}^i \in [x^i_0 x^i_1 x^i_2
\ldots x^i_n]$, we see the possible number of options for the next
entry $x^i_{n+1}$. If more than one option is available, we find a
sequence $\bar{y}^i$ with this entry replaced, or otherwise move
to the next entry. We need to continue this process for only the
next $M$ entries, since at least one of the entries $x^i_{n+l}$
and $y^i_{n+l}$ will be different for $l=1,2,\ldots,M$. Otherwise,
if each entry $x^i_{n+1},x^i_{n+2},\ldots, x^i_{n+M}$ has a unique
option, it means that there is only one allowed word of length
$M$. This will contradict the sensitivity of $\sigma$. Thus, we
find sequences $\bar{y}^1,\bar{y}^2,\ldots,\bar{y}^r$ such that
for each $i$, at least one of the entries $x^i_{n+l}$ and
$y^i_{n+l}$ are different for $l=1,2,\ldots,M$.

Therefore, for each $i$, $D(\sigma^n(\bar{x}^i),
\sigma^n(\bar{y}^i)) \geq \frac{1}{2^M}$.

Construct the set $C :=\{\bar{z}^1,\bar{z}^2,\ldots,\bar{z}^r\}$ as,

\centerline{$\bar{z}^i = \left\{%
\begin{array}{ll}
    \bar{y}^i, & \hbox{$D(\sigma^n(\bar{x}^1), \sigma^n(\bar{y}^i)) \geq \frac{1}{2^{M+1}}$;} \\
    \bar{x}^i, & \hbox{$\mbox{otherwise}$.} \\
\end{array}%
\right.$}

Therefore, $\sigma^n(\bar{x}^1)$ is at least $\frac{1}{2^{M+1}}$
apart from $\sigma^n(\bar{z}^i)$, $i=1,2,\ldots,r$. Thus,
$D_H(\overline{\sigma}^n(A), \overline{\sigma}^n(C)) \geq
\frac{1}{2^{M+1}}$, and hence $(\mathcal{K}(\Omega),
\overline{\sigma})$ is sensitive.

Conversely, let $(\mathcal{K}(\Omega), \overline{\sigma})$ be
sensitive with sensitivity constant $\delta$. For $\bar{x} \in
\Omega$ and for $\epsilon >0$,  let $S_{\epsilon}(\{\bar{x}\})$ be
an $\epsilon$-neighborhood of $\{\bar{x}\} \in
\mathcal{K}(\Omega)$. Then, there exists $A \in
S_{\epsilon}(\{\bar{x}\})$ and $n \in \mathbb{N}$ such that
$D_H(\overline{\sigma}^n(\{\bar{x}\}), \overline{\sigma}^n(A))
> \delta$. Thus, there exists a sequence $\bar{y} \in A \subset S_{\epsilon}(\bar{x})$ such
that $D_H(\overline{\sigma}^n(\{\bar{x}\}),
\overline{\sigma}^n(\{\bar{y}\})) > \delta$. Consequently, we get
$D(\sigma^n(\bar{x}),\sigma^n(\bar{y}))
> \delta$ establishing the sensitivity of $\sigma$.
\end{proof}

\begin{Remark}
It is known that when the system $(X,f)$ is cofinitely sensitive,
then so is $(\mathcal{K}(\Omega), \overline{\sigma})$ and vice-versa
\cite{sn2}. And since sensitive subshifts of finite type are
cofinitely sensitive, the above result should hold as a special case
of the general result. We however prove this in terms of sequences,
without referring to the cofinite sensitivity of the space.
\end{Remark}

\subsection{ We now consider the case when $(\Omega, \sigma)$ is
any subshift of $(A^{\mathbb{N}}, \sigma)$,} \textbf{where
$\mathcal{A}$ can be any alphabet set.}

The system $(\Omega, \sigma)$ induces the subsystem
$(\mathcal{K}(\Omega), \overline{\sigma})$.

\begin{Proposition}
If $(\Omega, \sigma)$ has dense set of periodic points, then so does
$(\mathcal{K}(\Omega), \overline{\sigma})$.
\end{Proposition}
\begin{proof}
Let $K \in \mathcal{K}(\Omega)$ and let $S_{D_H}(K, \epsilon)$ be a
neighborhood of $K$ in the hyperspace. As finite set of points in
$\Omega$ are dense in $\mathcal{K}(\Omega)$, there exists such a
finite set $ \{\bar{x}^1, \bar{x}^2, \ldots, \bar{x}^n \} \in
S_{D_H}(K, \epsilon)$, where each $\bar{x}^i=x^i_0 x^i_1 x^i_2
\ldots$. Further, there exists $\eta >0$ such that
$S_{D_H}(\{\bar{x}^1, \bar{x}^2, \ldots, \bar{x}^n \}, \eta) \subset
S_{D_H}(K, \epsilon)$. As periodic points are dense in $(\Omega,
\sigma)$, $\bar{y}^i = x^i_0 x^i_1 \ldots x^i_n y_{n+1} y_{n+2}
\ldots y_{n_i} x^i_0 x^i_1 \ldots x^i_n y_{n+1} y_{n+2} \ldots
y_{n_i} \ldots$ is a periodic point of period $k_i$ in $S(x_i,
\eta)$. Let $K = lcm \{ k_1, k_2, \ldots k_n \}$. Then, $\{
\bar{y}^1,\bar{y}^2, \ldots \bar{y}^n \}$ is a periodic point of
$\overline{\sigma}$ with period $K$, contained in
$S_{D_H}(\{\bar{x}^1, \bar{x}^2, \ldots, \bar{x}^n \}, \eta)$. Thus,
periodic points are dense in $(\mathcal{K}(\Omega),
\overline{\sigma})$.
\end{proof}

The result proved above is a manifestation of the known result in
the general case. It is known that the system $(\Omega, \sigma)$ can
be seen as a subsystem of $(\mathcal{K}(\Omega),
\overline{\sigma})$. And by expanding the symbol set (of the
original system) to an appropriate cardinality the system
$(\mathcal{K}(\Omega), \overline{\sigma})$ can be embedded into the
symbolic space of these new symbols. Then if $(\mathcal{K}(\Omega),
\overline{\sigma})$ is transitive, under successive iterations, any
two regions of the space labelled by distinct symbols interact. The
same holds for the original set of symbols. Consequently, any two
regions of the original dynamical system $(\Omega, \sigma)$ interact
under iterations of the map $\sigma$ making the $(\Omega, \sigma)$
transitive.

Similarly, we observe that

\begin{Proposition}
$(\Omega, \sigma)$ is weakly mixing if and only if
$(\mathcal{K}(\Omega), \overline{\sigma})$ is weakly mixing.
\end{Proposition}
\begin{proof}
Let $(\Omega, \sigma)$ be weakly mixing. Let $S_{D_H}(A_1,
\epsilon)$, $S_{D_H}(A_2, \epsilon)$ and $S_{D_H}(A_3, \epsilon)$,
$S_{D_H}(A_4, \epsilon)$ be two pairs of open sets in the
hyperspace. As finite sets are dense in the hyperspace, let
$\{\bar{a}^{i_1},\bar{a}^{i_2}, \ldots, \bar{a}^{i_k} \} \in
S_{D_H}(A_i, \epsilon)$, $1 \leq i \leq 4$. Further, there exists
$\eta >0$ such that $S_{D_H}(\{\bar{a}^{i_1}, \bar{a}^{i_2}, \ldots,
\bar{a}^{i_k} \}, \eta) \subset S_{D_H}(A_i, \epsilon)$,  $1 \leq i
\leq 4$. As $\sigma$ is weakly mixing, there exists $\bar{b}^{i_j} =
a^{i_j}_0 a^{i_j}_1 \ldots a^{i_j}_k b^{i_j}_{k+1} b^{i_j}_{k+2}
\ldots$, \ $\bar{b}^{i_j} \in S(\bar{a}^{i_j}, \eta)$, $i=1,2$,
$j=1,2, \ldots k$ and $n \in \mathbb{N}$ such that
$\sigma^n(\bar{b}^{i_j}) \in S(\bar{a}^{{(i+2)}_j}, \eta)$. Thus,
$B_i = \{\bar{b}^{i_1},\bar{b}^{i_2}, \ldots, \bar{b}^{i_k} \} \in
S_{D_H}(A_i, \epsilon)$ such that $\overline{\sigma}^n(B_i) \in
S_{D_H}(A_{i+2}, \epsilon)$ for $i=1,2$. Thus,
$(\mathcal{K}(\Omega), \overline{\sigma})$ is weakly mixing.

The converse can be deduced similarly as discussed above.
\end{proof}

\begin{Proposition}
$(\Omega, \sigma)$ is topological mixing if and only if
$(\mathcal{K}(\Omega), \overline{\sigma})$ is topological mixing.
\end{Proposition}
\begin{proof}
Let $(\Omega, \sigma)$ be topologically mixing. Let $S_{D_H}(A_1,
\epsilon)$, $S_{D_H}(A_2, \epsilon)$ be non empty open sets in the
hyperspace. As finite sets are dense in $\mathcal{K}(\Omega)$, let
$\{\bar{a}^{i_1},\bar{a}^{i_2}, \ldots, \bar{a}^{i_k} \} \in
S_{D_H}(A_i, \epsilon)$, $i=1,2$. Further, there exists $\eta >0$
such that $S_{D_H}(\{\bar{a}^{i_1}, \bar{a}^{i_2}, \ldots,
\bar{a}^{i_k} \}, \eta) \subset S_{D_H}(A_i, \epsilon)$. As $\sigma$
is topological mixing, for each $j$, there exists $n_{1j} \in
\mathbb{N}$ such that for any $n \geq n_{1j}$, there exists
$\bar{b}^{1_j}_n = ({a^{1_j}_0} {a^{1_j}_1}{ a^{1_j}_2} \ldots
{a^{1_j}_{k_n}} {b^{1_j}_{k_n +1}} \ldots ) \in  S(\bar{a}^{1_j},
\eta)$ such that $\sigma^n(\bar{b}^{1_j}_n) \in S(\bar{a}^{2_j},
\eta)$. Let $r = \max \{n_{1_j} : 1 \leq j \leq k \}$.  Thus, for
each $n \geq r$, there exists $\bar{b}^{1_j}_n = ({a^{1_j}_0}
{a^{1_j}_1}{ a^{1_j}_2} \ldots {a^{1_j}_{k_n}}{ b^{1_j}_{k_n +1}}
\ldots ) \in S(\bar{a}^{1_j}, \eta)$ such that
$\sigma^n(\bar{b}^{1_j}_n) \in S(\bar{a}^{2_j}, \eta)$.

Thus, since $\{ \bar{b}^{1_j}_n : 1 \leq j \leq k \} \in
S_{D_H}(\{\bar{a}^{1_1}, \bar{a}^{1_2}, \ldots, \bar{a}^{1_k} \},
\eta)$,  this ensures $\overline{\sigma}^n(S_{D_H}(A_1, \epsilon))
\bigcap S_{D_H}(A_2, \epsilon) \neq \phi$. As the above can be done
for all $n \geq r$, $(\mathcal{K}(\Omega), \overline{\sigma})$ is
topological mixing.

The converse can be deduced similarly as above.
\end{proof}

\begin{Remark}
Similar techniques have been used to prove the above result in a
much more general setting \cite{ba,ro,sn1}. The result below also
uses techniques similar to the general form as in \cite{sn2}.
\end{Remark}

Whereas for sensitivity,
\begin{Proposition}
$(\Omega, \sigma)$ is cofinitely sensitive if and only if
$(\mathcal{K}(\Omega), \overline{\sigma})$ is cofinitely sensitive.
\end{Proposition}
\begin{proof}
Let $(\Omega, \sigma)$ be cofinitely sensitive. We prove  that
$(\mathcal{K}(\Omega), \overline{\sigma})$ is cofinitely sensitive,
by showing that $\overline{\sigma}$ is cofinitely sensitive on the
set of all finite sets in $\mathcal{K}(\Omega)$.

It will suffice to prove the result in case of the alphabet set
$\mathcal{A}$ being discrete. In case $\mathcal{A}$ is equipped
with some metric $d$, the sensitivity constants can be suitably
modified.

Without loss of generality, let $\sigma$ be cofinitely sensitive
with sensitivity constant $\frac{1}{2^M}$. This implies that for any
neighborhood $[x_0 x_1 x_2\ldots x_n]$ of $\bar{x} = (x_0 x_1 x_2
\ldots)$, there exists $k_n \in \mathbb{N}$ such that for each $l
> k_n$, there exists $\bar{y}^l \in [x_0 x_1 x_2\ldots
x_n]$ such that $x_{l+1} x_{l+2} \ldots x_{l+M} \neq y^l_{l+1}
y^l_{l+2} \ldots y^l_{l+M}$.

Let $A = \{\bar{x}^1,\bar{x}^2,\ldots,\bar{x}^k\} \in
\mathcal{K}(\Omega)$. Let  $\mathcal{U} = <U_1, U_2, \ldots, U_k>$
be a neighborhood of  $A$ where $U_j = [x^j_0 x^j_1 x^j_2 \ldots
x^j_r]$. For each $j$ and each neighborhood $U_j$, there exists
$N_j$ such that for each $l> N_j$, $D(\sigma^l(\bar{x}^j),
\sigma^l(\bar{y}^{j \ l})) > \frac{1}{2^M}$, for some $\bar{y}^{j
\ l} \in U_j$.

Let $N = \max\{ N_1,N_2, \ldots, N_k\}$. For each $j$ and for each
$l \geq N$, there exists $\bar{y}^{j \ l} \in U_j$ such that
$x_{l+1} x_{l+2} \ldots x_{l+M} \neq y^{j \ l}_{l+1} y^{j \
l}_{l+2} \ldots y^{j \ l}_{l+M}$.

Therefore, for each $j$, for each $l > N$,
$D(\sigma^{l}(\bar{x}^j), \sigma^{l}(\bar{y}^{j \ l})) \geq
\frac{1}{2^M}$.

Construct the set $C_l :=\{\bar{z}^1,\bar{z}^2,\ldots,\bar{z}^k\}$
as,

\centerline{$\bar{z}^i = \left\{%
\begin{array}{ll}
    \bar{y}^{i \ l}, & \hbox{$D(\sigma^n(\bar{x}^1), \sigma^n(\bar{y}^{i \ l})) \geq \frac{1}{2^{M+1}}$;} \\
    \bar{x}^i, & \hbox{$\mbox{otherwise}$.} \\
\end{array}%
\right.$}

Therefore, $\sigma^l(\bar{x}^1)$ is at least $\frac{1}{2^{M+1}}$
apart from $\sigma^l(\bar{z}^i)$, $i=1,2,\ldots,k$. Thus,
$D_H(\overline{\sigma}^l(A), \overline{\sigma}^l(C_l)) \geq
\frac{1}{2^{M+1}}$. As such sets can be constructed for all $l >
N$, $(\mathcal{K}(\Omega), \overline{\sigma})$ is cofinitely
sensitive.

Conversely, Let $\bar{x} \in \Omega$ and let $\epsilon >0$ be
given. As $(\mathcal{K}(\Omega), \overline{\sigma})$ is cofinitely
sensitive with sensitivity constant $\delta$ and
$S_{\epsilon}(\{\bar{x}\})$ is $\epsilon$-neighborhood of
$\{\bar{x}\}$, there exists $n_0 \in \mathbb{N}$ such that for
each $n \geq n_0$, there exists $A_n \in
S_{\epsilon}(\{\bar{x}\})$ such that
$D_H(\overline{\sigma}^n(\{\bar{x}\}), \overline{\sigma}^n(A_n)) >
\delta$. Consequently, there exists a sequence $\bar{y_n} \in A_n
\subseteq S_{\epsilon}(\bar{x})$ such that
$D_H(\overline{\sigma}^n(\{\bar{x}\}),
\overline{\sigma}^n(\{\bar{y_n}\})) > \delta$. As the existence of
the set $A_n$ is guaranteed for all $n \geq n_0$, for each such
integer $n$, we obtain $\bar{y_n} \in S_{\epsilon}(\bar{x})$ such
that $D(\sigma^n(\bar{x}),\sigma^n(\bar{y_n}))
> \delta$. Hence $(\Omega, \sigma)$ is cofinitely sensitive.
\end{proof}

In case of transitive subshifts, cofinite sensitivity and syndetic
sensitivity turn out to be equivalent. Though, this equivalence
holds for any alphabet set, we prove it below only for the case of
discrete alphabet.

\begin{Proposition}
Let $(\Omega, \sigma)$ be a transitive subshift. Then, $(\Omega,
\sigma)$ is cofinitely sensitive if and only if $(\Omega, \sigma)$
is syndetically sensitive.
\end{Proposition}

\begin{proof}
Since every cofinitely sensitive map is syndectically sensitive, we
only need to prove that $(\Omega, \sigma)$ is cofinitely sensitive
whenever $(\Omega, \sigma)$ is syndetically sensitive.

 Let $(\Omega, \sigma)$ be syndetically sensitive and let
$\bar{x}=(x_0 x_1 x_2 x_3 \ldots x_n \ldots) \in \Omega$ be the
point with dense orbit. Let $K$ be the bound for syndetic
sensitivity of the cylinder $[x_0 x_1 x_2 x_3 \ldots x_n]$. It can
be seen that $K$ is the bound for syndetic sensitivity of any
other cylinder of the form $[x_{m+1} x_{m+2}  \ldots x_{m+r}]$.

Let  $\bar{y} \in Orb(\bar{x})$ and let $[y_0 y_1 y_2 \ldots y_r]$
be a neighborhood of $\bar{y}$. As the bound for syndetic
sensitivity for any cylinder around $\bar{y} \in Orb(\bar{x})$
must be $K$, there exists some $N \in \mathbb{N}$ such that for
each instant $l
>N$, there exists $\bar{z}^l \in [y_0 y_1 y_2\ldots y_r]$ such that
$y_{l+1} y_{l+2} \ldots y_{l+K} \neq z^l_{l+1} z^l_{l+2} \ldots
z^l_{l+K}$. Consequently, $D(\sigma^l(\bar{y}),
\sigma^l(\bar{z}^l)) \geq \frac{1}{2^K}$. Thus $(\Omega, \sigma)$
is cofinitely sensitive at each point of the orbit of $\bar{x}$
with the same sensitivity constant. As ${Orb(\bar{x})}$ is dense,
$(\Omega, \sigma)$ is cofinitely sensitive.
\end{proof}

\begin{Cor}
Let $(\Omega, \sigma)$ be a transitive subshift. $(\Omega,
\sigma)$ is syndetically sensitive if and only if
$(\mathcal{K}(\Omega), \overline{\sigma})$ is syndetically
sensitive.
\end{Cor}

One of the important examples of  subshifts are the \textit{Sturmian
subshifts }( see \cite{subru}). They are minimal shifts and it has
been shown in \cite{subru} that Sturmian subshifts are syndetically
sensitive.

\begin{Cor} All Sturmian subshifts are cofinitely sensitive. And
hence, the corresponding induced map $\overline{\sigma}$ on the
hyperspace of its compact subsets is also cofinitely sensitive.
\end{Cor}

The above corollary contradicts the result in \cite{subru} which
states that no Sturmian subshift is cofinitely sensitive. The
error there is a trivial one (follows from the definition there).

There exist subshifts $(\Omega, \sigma)$ such that
$(\mathcal{K}(\Omega), \overline{\sigma})$ is not sensitive
\cite{ishs}. The example constructed there  is not syndetically
sensitive, and hence not cofinitely sensitive. Does that mean all
subshifts $(\Omega, \sigma)$, for which $(\mathcal{K}(\Omega),
\overline{\sigma})$ is sensitive, need to be either syndetically
sensitive or cofinitely sensitive. We answer this in the negative,
by giving an example (from \cite{subru}) of a subshift $(\Omega,
\sigma)$ which is not syndetically sensitive, but for which
$(\mathcal{K}(\Omega), \overline{\sigma})$ happens to be
sensitive.

\begin{ex}
Let $(n_k)$ be a strictly increasing sequence of natural numbers.
Inductively, define a sequence of words $w_k$ as,

$$w_1 =0, w_2 = 1^{n_1}, w_3 = w_1w_2,  \ldots, w_{2k} = 1^{n_k},
w_{2k+1} = w_1 w_2 w_3 \ldots w_{2k},  \ \ldots $$

Put $\bar{x}=w_1 w_2 w_3 \ldots w_k \ldots \in \Sigma =
\{0,1\}^\mathbb{N}$

Clearly, $\bar{x}$ is a recurrent point for the dynamical system
$(\Sigma, \sigma)$, where $\sigma$ denotes the left shift on the
sequence space of two symbols.

\centerline{Define $X := \overline{Orb(\bar{x})}$ where,
$Orb(\bar{x}) = \{ \sigma^n(\bar{x}) : n \geq 0 \}$}

It has been shown in \cite{subru}, that $X$ defined above is
sensitive, but not syndetically sensitive. We here prove that
$(\mathcal{K}(X), \overline{\sigma})$ is sensitive.

 We first show that elements of the form $w_1 w_2
\ldots w_r 1^{\infty}$ are limit points of the orbit of $\bar{x}$
under the map $\sigma$. It is sufficient to show that all elements
of the form $w_1 w_2 \ldots w_{2k-1} 1^{\infty}$ are limit points
of the orbit.

It can be seen that $w_{2k+1}$ is the word $w_1 w_2 \ldots
w_{2k-1} 1^{n_k}$. Similarly $w_{2k+3}$ contains the word $w_1 w_2
\ldots w_{2k-1} 1^{n_k} 1^{n_{k+1}}$, $w_{2k+5}$ contains the word
$w_1 w_2 \ldots w_{2k-1} 1^{n_k} 1^{n_{k+1}} 1^{n_{k+2}}$ and so
on. Consequently, all the sequences of the form $w_1 w_2 \ldots
w_{2k-1} 1^{\infty}$ are limit points of the orbit and hence every
sequence of the form $w_1 w_2 \ldots w_r 1^{\infty}$ is a limit
point of the orbit.

As $Orb(\bar{x})$ is dense in $X$, the set $\mathcal{D} = \{
\{\bar{x}^1, \bar{x}^2, \ldots, \bar{x}^r\} : \bar{x}^i \in
Orb(\bar{x})$ for all $i, \ r \geq 1 \}$ is dense in
$\mathcal{K}(X)$. We shall, equivalently, show that
$\overline{\sigma}$ is sensitive on $\mathcal{D}$.

Let $\{\bar{x}^1, \bar{x}^2, \ldots , \bar{x}^k\} \in
\mathcal{D}$. Let $<U_1, U_2, \ldots, U_k>$ be its neighborhood,
where each $U_i$ is a neighborhood of $\bar{x}^i$. As $\bar{x}^i =
\sigma^{k_i}(\bar{x})$, let $[x^i_0 x^i_1 x^i_2 \ldots x^i_j
w_{n_i}]  \subseteq U_i$ for some $n_i$. Let $N = \max\{n_i : 1
\leq i \leq k\}$. Then, there exists $\bar{y}^i = x^i_0 x^i_1
x^i_2 \ldots w^N 1^{\infty} \in U_i$ for each $i$. Consider the
set $B = \{\bar{y}^1, \bar{y}^2, \ldots, \bar{y}^k\} \in <U_1,
U_2, \ldots, U_k>$.  The set $B$ reduces to the constant sequence
$111111 \ldots$ after some finite iterations. Thus, there exists
$l \in \mathbb{N}$ such that $D_H(\overline{\sigma}^l(A),
\overline{\sigma}^l(B)) \geq \frac{1}{2}$. Consequently,
$\overline{\sigma}$ is sensitive on $\mathcal{D}$.
\end{ex}

And for expansivity, we have
\begin{Proposition}
For any alphabet set $\mathcal{A}$, $(\mathcal{K}(\Omega),
\overline{\sigma})$ is expansive implies $(\Omega, \sigma)$ is
expansive. However, the converse does not hold in general.
\end{Proposition}
\begin{proof}
Let $(\mathcal{K}(\Omega), \overline{\sigma})$ be
$\delta$-expansive and let $\bar{x} ,\bar{y} \in \Omega$. Then, as
$\overline{\sigma}$ is expansive, for $\{x\}, \{y\} \in
\mathcal{K}(\Omega)$, there exists $k \in \mathbb{N}$ such that
$D_H ( \overline{\sigma}^k(\{\bar{x}\}),
\overline{\sigma}^k(\{\bar{y}\})) \geq \delta$. Consequently,  $D(
\sigma^k(\bar{x}),\sigma^k(\bar{y})) \geq \delta$. Thus, $(\Omega,
\sigma)$ is also $\delta$-expansive.

We now provide an example to show that the converse is not true.

Let $\Sigma$ be the sequence space of two symbols $0$ and $1$ and
let $\mathcal{K}(\Sigma)$ be the hyperspace of all non empty
compact subsets of $\Sigma$. It can be easily observed that
$(\Sigma, \sigma)$ is expansive with expansivity constant
$\frac{1}{2}$. However, we prove that the system
$(\mathcal{K}(\Sigma), \overline{\sigma})$ is not expansive.

Let if possible, $(\mathcal{K}(\Sigma), \overline{\sigma})$ be
expansive with expansivity constant $\delta$. Let $ n \in
\mathbb{N}$ such that $\frac{1}{2^n} < \delta$.

Let $S_1$ be the set of all sequences comprising of all $0$'s except
one string of $1$'s of length $r$, $ 0 \leq r \leq n$. Let $S_2$ be
the set of all sequences comprising of all $0$'s except one string
of $1$'s of length $r$, $0 \leq r \leq n+1$. Then, $D_H(S_1, S_2) =
\frac{1}{2^{n+1}}$. Also, $\overline{\sigma}(S_i) = S_i$, $i= 1,2$.
Thus, for any $k \in \mathbb{N}$, $D_H( ( \overline{\sigma}^k(S_1),
\overline{\sigma}^k(S_2)) = D_H(S_1, S_2) = \frac{1}{2^{n+1}} <
\delta$ which contradicts the definition of $\delta$.

Thus, the system $(\mathcal{K}(\Sigma), \overline{\sigma})$ is not
expansive.
\end{proof}

A similar result for expansivity holds in case of a general
dynamical system $(X,f)$. See \cite{sn1} for details.



\end{document}